\documentclass[11 pt]{article}
\usepackage{amsmath, amssymb}
\usepackage{amsthm} 
\usepackage{dcpic,pictex}
\usepackage{enumerate}

\theoremstyle{plain} 
\newtheorem{theorem}{\indent\bf Theorem}
\newtheorem{lemma}[theorem]{\indent\bf Lemma}
\newtheorem{corollary}[theorem]{\indent\sc Corollary}

\newtheorem{proposition}[theorem]{\indent\bf Proposition}

\theoremstyle{plain} 

\newtheorem{remark}[theorem]{Remark}

\def \EE {\mathbb{E}}
\def \RR {\mathbb{R}}
\def \pie {\pi_{\mathbb{E}}}
\def \jpe {J_p^1\mathbb{E}}
\def \jpm {J_p^1M}

\def \jp {J_p^1}
\def \j2p {J_{2p}^2}
\def \j2 {J^2}
\def \cson {C^{\infty}}

\makeatother

\title{\uppercase{ Canonical Involution on Double Jet Bundles}}

\author{H\"{u}lya Kad{\i}o\u{g}lu \\ \small{Yildiz Technical University, Esenler, Istanbul TURKEY}, \\ 
{\small Department of Mathematics, Idaho State University, Pocatello, ID 83209-8085, USA}, \\ \textit{email:} hkadio@yildiz.edu.tr}
\date{}
\begin{document}
   \maketitle
   
\footnote{
{\textit{Correspond to: Hulya Kadioglu, Yildiz Technical University, Istanbul-TURKEY,  email: hkadio@yildiz.edu.tr}}} 

   \begin{abstract} 
	In this study, we generalize double tangent bundles to double jet bundles.  We present a secondary vector bundle structure on a 1-jet of a vector bundle. We show that 1-jet of a vector bundle carries two vector bundle structures, namely primary and secondary structures.  We also show that the manifold charts induced by primary and secondary structures belong to the same atlas. We prove that double jet bundles can be considered as a quotient of second order jet bundle. We show that there exists a natural involution that interchanges between primary and secondary vector bundle structures on double jet bundles.
    
   \end{abstract}    

\maketitle

 \section{Introduction}

  In general, there are two ways to define $k-$ jets. The first definition is based on using the sections of a fibered manifold. In this definition, a $k-$jet is an equivalence class determined by an equivalence relation $\backsim_k$. Two sections of a fibered manifold are called $k-$related by the relation $\backsim_k$, if they have the same Taylor polynomial expansion at the point $x$ truncated at order k. This definition usually leads to a geometric approach which is applied to study of systems of differential equations  ( we refer to \cite{Baker}, \cite{Barco}, \cite{Bocharov}, \cite{Krasil}, \cite{Wu}, and \cite{Saunders} for more details).
\newline

The second and the more general definition of jet bundles is based on using the functions from  $N$ to $M$, where $N$ and $M$ are smooth manifolds. In this definition, a $k-$ jet of a function $f$ at $x$ is an equivalence class defined by an equivalence relation $\backsim_k$.  The equivalence of two functions is defined by the same way as the equivalence of sections.  The collection of all $k-$jets is called a $k-$jet bundle. One particular case is when $N=\RR$. In this case, the jet bundle is called tangent bundle of higher order. This jet bundle possesses a certain kind of geometric structure, which is called almost tangent structure of higher order. More generally, if $N=\RR ^p$, then the jet bundle is called the tangent bundle of $p^k$ velocities. This concept was introduced by Ehresmann to develop classical field theory in autonomous sense \cite{Leon}.
\newline

Throughout this paper, we will be considering the second definition: 1-jets with source at the origin of $\RR ^p$, and targeted in $M$. In section 2, we give some preliminary theorems that we need for the paper. In section 3, we study  1-jet of an arbitrary vector bundle and present a secondary vector bundle structure on a 1-jet of the vector bundle. Due to two vector bundle structures, 1-jet of a jet bundle is considered as a double vector bundle (DVB). In terms of the induced manifold structures on the total space of DVB, we prove that both belongs to the same atlas. We also prove that 1-jet of a jet bundle (double jet bundle) can be considered as a quotient of a second order jet bundle. Moreover, we show that two vector bundle structures are isomorphic on a double jet bundle by defining a canonical involution that interchanges between such structures. In Appendix section (Section 4), we prove some identities and statements that we used throughout the paper.


\section{Preliminaries}

 In this section we summarize some necessary preliminary materials that we need for a self contained presentation of our paper.

\subsection{ First Order Jets: $\jpm $}

Let $\cson (\RR ^p)$ be the algebra of $\cson $ -functions on the Euclidean space $\RR ^p$ with natural coordinates $(u_1,u_2,...,u_p)$. Let $f,g \in \cson (\RR ^p)$; $f$ is equivalent to $g$ if $f(0)=g(0)$ and $\partial /\partial u_i (f)=\partial /\partial u_i (g)$ at $u=0 \in \RR ^p$ for every $i=1,2,...,p$; clearly, this is an equivalence relation. 

Now, let $M$ be an $m$ dimensional manifold. Consider the set $C^{\infty}(\RR ^p;M)$ of all smooth maps $\phi:\RR ^p \to M$, and take elements $\phi,\xi \in C^{\infty}(\RR ^p;M)$ . Then $\phi$ is equivalent to $\xi$ if $f \circ \phi$ is equivalent to $f \circ \xi$ for every $f \in \cson (M)$.This is again an equivalence relation; and denoted by $j^1(\phi)$ the equivalence class of $\phi \in \cson (\RR ^p;M)$, and called a 1-jet in $M$ at $\phi (0)$. We denote $\jpm $ the set of all equivalence classes in $\cson (\RR ^p;M)$. If $(U, x_1,...x_m)$ is a local chart in $M$, then $(J_p^1U, x_1,...,x_m,x_{\alpha}^1,....x_{\alpha}^m)$ is the local chart for $\jpm $, with $\alpha=1,2,...,p$ by

\begin{eqnarray}
x_i(j^1(\phi))=x_i(\phi(0)) \nonumber \\ 
x_{\alpha}^i(j^1(\phi))=\frac{\partial(x_i \circ \phi)}{\partial u_{\alpha}}|_0. \label{equation1}
\end{eqnarray}

Let $j^1\phi$ be an arbitrary point, and let $\phi_{\alpha}:\mathbb{R} \to M$ be the differentiable curve given by $\phi_{\alpha}(u)=\phi(0,...,u,...0)$, with $u$ at the $\alpha^{th}$ place; then, associated to $j^1\phi$ there is a unique $(p+1)$-tuple $[x;X_1,...,X_p]$ given by 

\begin{equation}
x=\phi(0) \hspace{2cm} X_{\alpha}=\phi^{\alpha}(\frac{d}{du}|_0), \nonumber 
\end{equation} 

where $d/du$ is the canonical vector field tangent to $\RR $. From now on, we shall write $[x;X_1,...,X_p]$ simply as $[x;X_{\alpha}]$ and shall identify $j^1\phi \equiv [x;X_{\alpha}]$ if there is no confusion.

\begin{remark} \nonumber
In later sections, we sometimes use the notation $X_{\alpha}$ for $ \frac{\partial \Phi}{\partial u_{\alpha}}|_0 $.
\end{remark}

Now we focus on some functorial properties of jet bundles.

\begin{theorem} \label{corderoo} \cite{Cordero}

\begin{enumerate}[(i)]
\item If $h:M \to N$ is a differentiable function, then $h$ induces a canonical differentiable map $h^1: \jpm \to J_p^1N$ given by

\begin{equation*}
h^1(j^1\phi)=j^1(h\circ \phi),\hspace{2cm} \forall j^1\phi \in \jpm 
\end{equation*}

and in terms of previous identification, we have, $h^1([x;X_{\alpha}])=[h(x);h_{\ast}X_{\alpha}]$. 

\item If $h$ is a diffeomorphism, then the induced map is also a diffeomorphism and, moreover $(h^1)^{-1}=(h^{-1})^1$. Also for the manifolds $M$ and $N$, $\jp (M \times N)$ is diffeomorphic to $\jp M \times \jp N$.

\item If $M$ is a real vector space of dimension $m$, then $\jpm $ inherits a vector space structure: for any $j^1f ,j^1g \in \jpm $ and $\lambda \in \RR $, operations are given by

\begin{equation*}
j^1f +j^1g=j^1(f+g), \hspace{2cm} \lambda j^1f =j^1(\lambda f), 
\end{equation*}

where $f+g$ and $\lambda f$ are defined in the usual way. Vector space operations of $\jpm $ as $[x;X_{\alpha}]+\lambda [y;Y_{\alpha}]=[x+\lambda y;X_{\alpha}+\lambda Y_{\alpha}]$.

\end{enumerate}
\end{theorem}

\subsection{Second Order Jets: $J_{2p}^2 M$}

Now we consider the $2$- jets by taking classes having equivalence up to all derivatives of second order. The natural atlas of this bundle can be obtained as follows:
\newline

Let $\phi, \phi' \in C^{\infty}(\RR^{2p}, M)$. We say $\phi$ is equivalent to $\phi'$ if 

\begin{eqnarray*}
\phi(0,0)&=&\phi'(0,0) \nonumber \\
\frac{\partial \phi}{\partial \bar{u}_{\bar{\alpha}}}|_{(0,0)}&=& \frac{\partial \phi'}{\partial \bar{u}_{\bar{\alpha}}}|_{(0,0)} \nonumber \\
\frac{\partial ^2 \phi}{\partial \bar{u}_{\bar{\alpha}} \partial \bar{u}_{\bar{\beta}}}|_{(0,0)}&=&\frac{\partial ^2 \phi'}{\partial \bar{u}_{\bar{\alpha}} \partial \bar{u}_{\bar{\beta}}}|_{(0,0)}. \hspace{2cm} 1\leq \bar{\alpha}, \bar{\beta} \leq 2p \nonumber
\end{eqnarray*}

Let $E$ be a finite dimensional real vector space, then $J_{2p}^2 E$ is regarded as a finite dimensional real vector space. Let $L(\RR ^{2p}, E)$ denote the vector space of linear functions $A: \RR ^{2p} \to E$ and let $S_2(\RR ^{2p};E)$ denote the vector space of all symmetric bi-linear functions $B: \RR ^{2p} \times \RR ^{2p} \to E$. The function 

\begin{equation*}
j^2 \phi \rightarrow (\phi(0,0), \frac{\partial \phi}{\partial \bar{u}_{\bar{\alpha}}}|_{(0,0)}, \frac{\partial ^2 \phi }{\partial \bar{u}_{\bar{\alpha}} \partial \bar{u}_{\bar{\beta}}}|_{(0,0)}), \hspace{1.5cm} 1\leq \bar{\alpha},\bar{\beta} \leq 2p 
\end{equation*}

is a canonical isomorphism, where 

\begin{equation*}
\bar{u}_{\bar{\alpha}}=\left\{ \begin{array}{rcl} u_{\alpha} \hspace{1.5cm} \bar{\alpha}=1,2,...,p,\\ w_{\alpha}  \hspace{1cm}  \bar{\alpha}= p+1,....,2p.  \end{array}\right.
 \end{equation*}

We note that we identify $\RR ^{2p}$ with $\RR ^p \times \RR ^p$. Let $\varphi: U \to E$ be a local chart of $M$ that, without loss of generality, we assume maps onto a vector space $E$. Let $(x,A,B) \in J_{2p}^2(E) $; the corresponding 2-jet is the one represented by 

\begin{equation*}
\phi (\bar{u})=x+A.\bar{u}+\frac{1}{2} \bar{u}^T.B.\bar{u}.
\end{equation*}

In terms of curve notations discussed for the first order jets, we may identify $ j^2 \phi \in J^2_{2p}$ with the triple

\begin{equation*}
j^2 \phi=[x;A_{\bar{\alpha}}; B_{\bar{\alpha} \bar{\beta}}] , \hspace{2cm} 1\leq \bar{\alpha}, \bar{\beta} \leq 2p
\end{equation*}

where 
\begin{equation*}
A_{\bar{\alpha}}=\frac{\partial \phi}{\partial \bar{u}_{\bar{\alpha}}}|_{(0,0)}
\end{equation*}
 
and 

\begin{equation*}
B_{\bar{\alpha}\bar{\beta}}=\frac{\partial ^2 \phi }{\partial \bar{u}_{\bar{\alpha}} \partial \bar{u}_{\bar{\beta}}}|_{(0,0)}.
\end{equation*}


\section{Jet Bundle to a Vector Bundle}

Let $\pie :\EE \to M$ be a vector bundle with the local bundle trivialization 

\begin{equation}
\psi:\pie ^{-1}(U) \to U \times E \nonumber
\end{equation}

where $U$ is an open subset of $M$. Then, $\jp \EE$ can be considered as a jet bundle on $\EE$ (the total space of the VB $\pi_{\EE}$). The bundle trivialization on $\jp \EE$ is 

\begin{equation}
\tilde{\psi}:\jp \EE \to \EE \times L(\RR ^p, \RR ^{m+k})   \nonumber
\end{equation}

by 

\begin{equation}
\tilde{\psi}(j^1\Phi)=(\Phi(0), \frac{\partial \Phi}{\partial u_{\alpha}}|_0) \nonumber
\end{equation}

where $\Phi \in C^{\infty}(\RR ^p, \EE)$. Then, $\Phi(0)=(\psi)^{-1}(x,y)$, and 

\begin{equation} 
\frac{\partial \Phi}{\partial u_{\alpha}}|_0=(\frac{\partial (\bar{x}_i \circ \Phi)}{\partial u_{\alpha}}|_0, \frac{\partial (\bar{y}_j \circ \Phi)}{\partial u_{\alpha}}|_0) \nonumber
\end{equation}

where, $\bar{x}_i =  x_i \circ \pie$,  $\bar{y}_j = y_j \circ pr_2 \circ \psi$, and $x_i, y_j$ are the local coordinate functions of $M$ and $E$ respectively. 
\newline

Using curve notation gives:

\begin{equation}
j^1 \Phi \equiv [x,y;X_{\alpha}, Y_{\alpha}]  \nonumber 
\end{equation}

where 

\begin{equation}
X_{\alpha}=\frac{\partial (\bar{x}_i \circ \Phi)}{\partial u_{\alpha}}|_0 \nonumber
\end{equation}

 and 

\begin{equation}
Y_{\alpha}=\frac{\partial (\bar{y}_j \circ \Phi)}{\partial u_{\alpha}}|_0. \nonumber 
\end{equation}

Moreover, for each smooth manifold $M$, $\jpm$ carries a vector bundle structure (see Lemma (\ref{theo1})). Since $\EE$ is a smooth manifold, then $\jpe$ carries a VB structure, where the intrinsic operations are given by the following:
\newline

For a local bundle chart $\tilde{\psi}:(\pie \circ \tilde{\pi})^{-1} \to \pie^{-1}(U) \times L(\RR ^p, \RR^{m+k})$, the mapping $+_1$ and $\bullet _1$  on $(\tilde{\pi})^{-1}\{(x,y)\}$ are defined locally as

\begin{equation}
j^1 \Phi +_1 j^1 \Phi'=(\tilde{\psi})^{-1}((\psi)^{-1}(x,y), (pr_2 \circ \tilde{\psi})(j^1 \Phi) +(pr_2 \circ \tilde{\psi})(j^1 \Phi')) \nonumber
\end{equation}

and

\begin{equation}
\lambda \bullet_1 j^1 \Phi=(\tilde{\psi})^{-1}((\psi)^{-1}(x,y), \lambda . (pr_2 \circ \tilde{\psi})(j^1 \Phi)).\nonumber 
\end{equation}

Using curve notation, one can see that

\begin{equation}
j^1 \Phi+_1 j^1 \Phi '=[x,y;X_{\alpha},Y_{\alpha}]+_1 [x,y;X'_{\alpha}, Y'_{\alpha}]=[x,y; X_{\alpha}+X'_{\alpha}, Y_{\alpha}+Y'_{\alpha}] \nonumber
\end{equation}

and

\begin{equation}
\lambda \bullet_1 [x,y;X_{\alpha}, Y_{\alpha}]=[x,y; \lambda . X_{\alpha}, \lambda . Y_{\alpha}] \nonumber
\end{equation}

where $U$ is an open subset of $M$.

\begin{remark}

Hereafter, we will refer to above vector bundle structure as the primary structure.

\end{remark}

\subsection{Secondary VB Structure on $\jp \EE$}

We recall that there exists an induced a canonical smooth function $\pi^1_{\EE}:\jp \EE \to \jpm$, where $\pi_{\EE}:\EE \to M$ a smooth bundle projection of $\EE $. One can easily prove that $(\pi_{\EE})^1$ is a surjective map. Moreover, it can be seen from Theorem \ref{corderoo} that $\jp E$ is a vector space (isomorphic to $L(\RR ^p, E)$). If we let the local  trivialization of $\EE$ be the map $\psi$, then  $\psi ^1:\jp (\pie ^{-1}(U)) \to \jp (U \times E)$ is a diffeomorphism.
\newline

Now we consider the trivialization domains: 

\begin{lemma} 
 
Let $\pie^{-1}(U)$ be a local trivialization domain of the bundle $\EE$. Then

\begin{equation} 
\jp (\pie^{-1}(U)) = (\pie ^1)^{-1}(\jp U) = (\pie \circ \tilde{\pi})^{-1}(U). \label{triviadomain}
\end{equation}  

\end{lemma}

\begin{proof}

To prove equation (\ref{triviadomain}), we will prove followings:
\newline 

\begin{enumerate}[(i)]
\item $\jp (\pie^{-1}(U)) \subset (\pie ^1)^{-1}(\jp U)$, 
\newline

\item $(\pie ^1)^{-1}(\jp U) \subset (\pie \circ \tilde{\pi})^{-1}(U)$, and 
\newline

\item $(\pie \circ \tilde{\pi})^{-1}(U) \subset \jp (\pie^{-1}(U))$.
\newline
\end{enumerate}

Suppose that 

\begin{equation}
j^1\phi \in \jp (\pie ^{-1}(U)). \nonumber
\end{equation} 

\noindent By its own definition, $\phi (0) \in \pie ^{-1}(U)$, then 

\begin{eqnarray}
\pie (\phi(0))\in U &\Rightarrow & j^1(\pie \circ \Phi) \in \jp U \nonumber \\ 
                    & \Rightarrow & \pie ^1(j^1 \Phi) \in \jp U \nonumber \\
                    &\Rightarrow & j^1 \phi \in (\pie ^1)^{-1}(\jp U) \nonumber 										
\end{eqnarray}										

Then $\jp ((\pie^1)^{-1}(U)) \subset (\pie ^1)^{-1}(\jp U)$. The first statement is proven.
\newline

Suppose that $j^1 \Phi \in (\pie ^1)^{-1}(\jp U).$ Then

\begin{eqnarray*}
\pie^1 (j^1 \Phi) \in \jp U &\Rightarrow &j^1 (\pie \circ \Phi) \in \jp U \\
                           &\Rightarrow & (\pie \circ \Phi)(0) \in U \\
													 & \Rightarrow & \Phi(0) \in (\pie)^{-1}(U) \\
													 &\Rightarrow & \tilde{\pi}^{-1}(\Phi(0)) \in \tilde{\pi}^{-1} (\pie ^{-1} (U)) \\
													 &\Rightarrow & j^1 \Phi \in (\pie \circ \tilde{\pi})^{-1}(U).
\end{eqnarray*}

Then  $(\pie ^1)^{-1}(\jp U) \subset (\pie \circ \tilde{\pi})^{-1}(U)$. The second statement is proven.

On the other hand, suppose that $j^1 \Phi \in (\pie \circ \tilde{\pi})^{-1}(U).$ Then

\begin{eqnarray}
(\pie \circ \tilde{\pi})(j^1 \Phi) \in U &\Rightarrow & (\pie (\tilde{\pi}(j^1 \Phi))) \in U  \nonumber \\
                                               &\Rightarrow & \pie(\Phi (0)) \in U \nonumber \\
																							 & \Rightarrow & \Phi (0) \in \pie^{-1} (U) \nonumber \\
																							 & \Rightarrow & j^1 \Phi \in \jp (\pie^{-1}(U)) \nonumber 
\end{eqnarray}

which shows that $(\pie \circ \tilde{\pi})^{-1}(U) \subset \jp (\pie^{-1}(U)). $ This completes the proof.

\end{proof}

\begin{proposition} \label{prop1}

Let $\pie :\EE \to M$ be a vector bundle. Then $\pie ^1 :\jpe \to \jpm$ is a vector bundle so that the manifold $\jpe $ has two vector bundle structures, namely , its primary vector bundle structure as the first jet bundle of manifold $\EE $, and a secondary structure with $\jpm $ as the base manifold. Moreover the induced charts on $\jpe$ from primary and secondary structures belong to the same atlas.

\end{proposition} 

\begin{proof}

First, we start the proof by showing that $\psi^1$  is the bundle trivialization of the secondary jet bundle $\jpe $:
\newline

Since $Pr_1 \circ \psi ^1 =\pie ^1$ , then  $\pie ^1: \jpe \to \jpm $ is a smooth fiber bundle. (Here $Pr_1:\jp U \times \jp E \to \jp U$ represents the first projection. 
We note that $\jp (U \times E)$ and $\jp U \times \jp E $ are diffeomorphic by setting $j^1 f \cong (j^1 (f_1), j^1(f_2))$, where $f=(f_1, f_2):\RR ^p \to U \times E$.)
\newline

Now we  identify $j^1\Phi \in \jpe$ as quadruple $[x,X_{\alpha};y,Y_{\alpha}]$ such that

\begin{equation}
X_{\alpha}=\frac{\partial (\pie \circ \Phi)}{\partial u_{\alpha}}|_0  \label{secondary1}
\end{equation}

and

\begin{equation}
Y_{\alpha}=\frac{\partial (pr_2 \circ \psi \circ \Phi)}{\partial u_{\alpha}}|_0.  \label{secondary2} 
\end{equation}

Showing that the fiber map being linear, proves that $\pie^1$ is a vector bundle over $\jpm$. 


Let $j^1 \Phi, j^1 \Phi ' \in (\pie ^1) ^{-1} \{ j^1 \theta \}$, for each $ j^1 \theta =[x;X_{\alpha}] \in \jpm $. Then  $j^1\Phi =[x,X_{\alpha};y,Y_{\alpha}]$  for $\alpha=1,2,...,p$, and $j^1\Phi' =[x,X_{\alpha};y',Y'_{\alpha}]$. The secondary vector bundle operations on this fiber are defined by:

\begin{eqnarray}
j^1 \Phi +_{_2}  j^1 \Phi '&=&(\psi^1)^{-1}(j^1 \theta, (pr_2 \circ \psi^1)(j^1\Phi)+(j^1 \theta, pr_2 \circ \psi^1)(j^1\Phi')) \nonumber \\
                           &=&[x,X_{\alpha};y+y',Y_{\alpha}+Y'_{\alpha}] \nonumber
\end{eqnarray}

and

\begin{eqnarray}
\lambda \bullet_{_2} j^1\Phi&=&(\psi^1)^{-1}(j^1\theta, \lambda. (pr_2 \circ \psi^1)(j^1\Phi)) \nonumber \\
                            &=& [x,X_{\alpha};\lambda y,\lambda Y_{\alpha}] \nonumber
\end{eqnarray}

where $j^1 \Phi,j^1 \Phi ' \in (\pie ^1)^{-1}\{j^1\theta\}$, and $\lambda \in \RR$. Thus,

\begin{eqnarray}
\psi^1_{[x,X_{\alpha}]}([x,X_{\alpha};y,Y_{\alpha}]+_{_2}\lambda \bullet_{_2} [x,X_{\alpha};y',Y'_{\alpha}]&=&\psi^1 ([x,X_{\alpha};y+\lambda y',Y_{\alpha}+\lambda Y'_{\alpha}])\nonumber \\
                                                                 &=&[y+\lambda y',Y_{\alpha}+\lambda Y'_{\alpha}] \nonumber \\
                                                                 &=&[y,Y_{\alpha}]+_{\EE} \lambda \bullet_{\EE} [y',Y'_{\alpha}] \nonumber																																 \end{eqnarray}

which equals to 
\begin{equation*}
\psi^1_{[x,X_{\alpha}]}([x,X_{\alpha}; y,Y_{\alpha}])+_2 \lambda \bullet_{\EE} \psi^1_{[x,X_{\alpha}]}([x,X_{\alpha}; y', Y'{\alpha}]) \nonumber
\end{equation*}                                                               

This shows that $\psi^1_{[x,X_{\alpha}]}$ is a linear function. Therefore, $\jpe $ is a vector bundle with its secondary structure.
\newline

By now,  we have shown that $\jpe $ carries two vector bundle structures (namely primary and secondary). One can easily see that these two structures define coordinate charts on the total space $\jpe $. Now we will show that these two charts belong to the same atlas. To do this we will show that identity map of $\jpe $ is a diffeomorphism between two VB structures based on $\jpm $ and $\EE $. 
\newline

Let $\EE $ has the local coordinate maps $\bar{x}_i, \bar{y}_j$ with $1\leq i \leq m,1\leq j \leq k$. Considering the primary structure on $\jpe$: For all $j^1\phi \in \jpe $, there exists triples $(x,y,\rho)$ with $\Phi(0)=\psi^{-1}(x,y)$ and $\rho:\RR ^p \to \RR ^{m+k}$ which is a linear function that corresponds to the matrix $A=\begin{bmatrix}\frac{\partial \Phi_{\bar{a}}}{\partial u^{\alpha}}|_0\end{bmatrix}$. Here the term $\Phi_{\bar{a}}$ is defined by 

\begin{equation*}
\Phi_{\bar{a}}=\left\{ \begin{array}{rcl} (\bar{x}_i \circ \Phi) \hspace{4cm} \bar{a}=1,2,...,m,\\ (\bar{y}_j \circ \Phi)\hspace{2.9cm} \bar{a}=m+1,....,m+k.  \end{array}\right.
\end{equation*}

Therefore the matrix $A$ consists of two sub matrices: they are

\begin{equation}
\begin{bmatrix}\frac{\partial (\bar{x}_i \circ \Phi)}{\partial u^{\alpha}}|_0\end{bmatrix} \hspace{2cm} and \hspace{2cm} \begin{bmatrix}\frac{\partial (\bar{y}_j \circ \Phi)}{\partial u^{\alpha}}|_0\end{bmatrix}.\label{matrix}
\end{equation}

On the other hand, considering secondary structure on $\jpe $,  given any $j^1\Phi \in \jpe $, there exists quadruple $(x,f;y,g)$, where $\Phi(0)=(x,y)$,  $f=\begin{bmatrix}\frac{\partial (\bar{x}_i \circ \Phi)}{\partial u^{\alpha}}|_0\end{bmatrix}$ and $g=\begin{bmatrix}\frac{\partial (\bar{y}_j \circ \Phi)}{\partial u^{\alpha}}|_0\end{bmatrix}.$
\newline
 
It is clear that the the matrix representations of $f$ and $g$ are defined the same as in Equation \ref{matrix}. Let $\Omega$ be the identity map of $\jpe $, $\Psi_{\EE}$, $\Psi_M$ and $\Psi_E$ be the local trivializations of jet bundles $\jpe,\jpm$ and $\jp E$ respectively , and $\varphi: U \subset M \to \RR ^m$ be a coordinate chart. Then the local form of $\Omega$ is given by the following commutative diagram:

\begin{equation*}
\begindc{\commdiag}
\obj(1,3)[A]{$(\pie^1)^{-1}(\jp U) $}
\obj(1,1)[C] {$\varphi(U) \times L(\RR ^p, \RR ^m) \times E \times L(\RR ^p, \RR ^k) $}
\obj(10,3)[B] {$(\pie \circ \tilde{\pi})^{-1} (U) $} 
\obj(10,1)[D] {$ \varphi(U) \times E \times L(\RR ^p,\RR ^m) \times L(\RR^p, \RR^k) $}
\mor{A}{C}{$\varphi_2$}
\mor{B}{D}{$\varphi_1$}
\mor{A}{B}{$\Omega$}
\mor{C}{D}{$\hat{\Omega}$}
\enddc
\end{equation*}

where 

\begin{equation*}
\varphi_1=(\varphi \times id_{E \times L(\RR^p, \RR^m) \times L(\RR ^p, \RR^k)}) \circ (\psi \times \xi) \circ \tilde{\psi},
\end{equation*} 

and

\begin{equation*}
\varphi_2=(\varphi \times id_{L(\RR^p, \RR^m) \times E \times L(\RR ^p, \RR^k) }) \circ (\Psi_M \times \Psi_E) \circ \psi^1
\end{equation*}

and 

\begin{eqnarray}
\xi:&L(\RR^p, \RR^{m+k}) \to& L(\RR^p, \RR^m) \times L(\RR^p, \RR^k) \nonumber \\
    &f_g  \to        & \xi(f_g)=(f, g) \nonumber 
\end{eqnarray}

where  $f_g (u)=(f(u),g(u))$.
\newline

\noindent From Lemma \ref{lemmaomega}, the local form $\hat{\Omega}$ is $(x,f,y,g) \to (x,y,f,g)$. It is clear that the local form is a surjective map. Since  $\hat{\Omega}=pr_1 \times pr_3 \times pr_2 \times pr_4$, then the local form is differentiable with its own inverse. Therefore $\hat{\Omega}$ is a diffeomorphism which implies that two structure charts belong to the same atlas. 

\end{proof}

\begin{remark}
The secondary structure defined on tangent bundles can be found in \cite{Fisher}.
\end{remark}


\subsection{ $\jp (\jpm) $ as a Quotient Manifold}

Let us consider the case $\EE = \jp M$ for the smooth manifold $M$. By Lemma \ref{theo1}, $\jp M$ is a vector bundle. By Proposition \ref{prop1}, $\jp (\jp M)$ has two vector bundle structures  both based on $\jpm$. To define the double jet manifold, we begin with the idea of smooth functions on $\jpm$. 
\newline

Let $\phi: \RR ^p \times \RR ^p \to M $ be a smooth function. We define 

\begin{eqnarray}
\Phi &: \RR ^p &\to \jpm \nonumber \\
     & u       & \rightarrow \Phi (u)=\jp (\phi_u) \label{Phi} 
\end{eqnarray}

where $\phi_u$ is defined by 

\begin{eqnarray*}
\phi_u&:\RR ^p &\to M \nonumber \\
      & v     & \rightarrow \phi_u (v)=\phi(u,v) 
\end{eqnarray*}

It can easily be seen that $\phi_u$ is a smooth function.
\newline

$$
\begindc{\commdiag}
\obj(1,4)[A]{$\RR ^p$}
\obj(5,4) [B] {$\jpm$}
\obj(5,1) [C] {$ M\times L(\RR ^p , \RR ^m )$}
\mor{A}{B}{$\Phi$}
\mor{A}{C}{$\psi \circ \Phi$}
\mor{B}{C}{$\psi$}
\enddc
$$

Here 
\begin{eqnarray*}
(\psi \circ \Phi)(u)&=&(\phi_u (0), \frac{\partial \phi_u}{\partial w_{\alpha}}|_{w=0}), \hspace{2cm} u\in \RR^p. \nonumber \\
                    &=&(\phi(u,0), \frac{\partial \phi_u}{\partial w_{\alpha}}|_{w=0})
\end{eqnarray*}

is a smooth function. Therefore, $\Phi \in C^{\infty}(\RR^p, \jpm).$ 

\begin{proposition} \label{lambda}

Let $\Phi$ and $\phi$ be the functions defined by Equation \ref{Phi}. Then $\jp (\jp M)$ is a quotient of $J_{2p}^2 M$, with the quotient map $\Lambda$ defined by the following: 

\begin{eqnarray*}
\Lambda:&  J_{2p}^2 (M) \to &\jp (\jpm) \nonumber \\
        &J^2 \phi \rightarrow & J^1 \Phi.
\end{eqnarray*}

\end{proposition}

\begin{proof}
 
Let $\psi$ be a local trivialization of the vector bundle $\jpm $, then considering local charts for $J_{2p} ^2 M $, and  $\jp (\jp M)$, we have the commutative diagram:
 
\begin{equation} \label{lambdatilde}
\begindc{\commdiag} 
\obj(1,4) [A] {$ J_{2p}^2 M $}
\obj(10,4) [B] {$\jp (\jp M)$}
\obj(1,1) [C] {$ M\times L(\RR ^{2p} , \RR ^m ) \times S_2(\RR ^{2p}, \RR ^m)$}
\obj(10,1) [D] {$ M \times L(\RR ^p , \RR ^m) \times L(\RR ^p,\RR ^m) \times L(\RR ^p, \RR ^{mp})$}
\mor{A}{B} {$\Lambda $} 
\mor{A}{C} {$\psi ^2$}
\mor{B}{D} {$\bar{\psi}  \circ \psi ^1=\widetilde{\psi} $}
\mor{C}{D} {$\hat{\Lambda}$} 
\enddc
\end{equation}
\newline 

\noindent where $\bar{\psi}$ denotes the local trivialization of $\jp (M\times L(\RR ^p, \RR ^m))$, and $\psi^2$ stands for the local trivialization of $J^2_{2p} M$.  We note that due to Proposition \ref{prop1}, it is also possible to use the induced chart of primary structure. 
\newline 

For any $A \in L(\RR ^{2p}, \RR ^m)$, we can identify  $A$ in terms of the sub matrices of form  $m \times p$ with $A=[A_{\alpha} A_{\dot{\alpha}}]$ where $A_{\alpha} , A_{\dot{\alpha}}\in L(\RR ^p , \RR ^m)$. Similarly, for any $B \in S_2(\RR ^{2p}, \RR ^m)$, we can identify $B$ as the sub matrices 

$$B_{\bar{\alpha}\bar{\beta}}=\begin{pmatrix}
B_{\alpha\beta} & B_{\dot{\alpha}\beta} \\
B_{\alpha \dot{\beta}} & B_{\dot{\alpha}\dot{\beta}}
\end{pmatrix}.
$$

Suppose that $(x,A,B) \in M \times L(\RR ^{2p}, \RR ^m)$. Let $(\psi^2)^{-1}(x,A,B)=j^2\phi$ with the following properties:
\newline

\begin{equation*}
x=\phi(0,0),\hspace{3mm} A=\frac{\partial \phi}{\partial \bar{u}_{\alpha}}|_{(0,0)},\hspace{3mm} [B_{\alpha \beta}]=\frac{\partial ^2 \phi}{\partial \bar{u}_{\alpha} \bar{u}_{\beta}}|_{(0,0)}. 
\end{equation*}

Therefore the sub matrices $A_{\alpha}$, $A_{\dot{\alpha}}$, $B_{\alpha \beta}$, $B_{\dot{\alpha} \beta}$, $B_{\alpha \dot{\beta}}$ and $B_{\dot{\alpha} \dot{\beta}}$ are

\begin{equation*}
A_{\alpha}=\frac{\partial \phi}{\partial u_{\alpha}}|_{(0,0)} \hspace{1cm} A_{\dot{\alpha}}=\frac{\partial \phi}{\partial w_{\alpha}}|_{(0,0)}
\end{equation*}

and

\begin{equation*} \nonumber
[B_{\alpha \beta}]=\frac{\partial ^2 \phi}{\partial u_{\alpha}u_{\beta}} \hspace{1cm} 
[B_{\dot{\alpha}\beta}]=\frac{\partial ^2 \phi}{\partial w_{\alpha}u_{\beta}} \hspace{1cm} 
[B_{\alpha \dot{\beta}}]=\frac{\partial ^2 \phi}{\partial u_{\alpha}w_{\beta}} \hspace{1cm} 
[B_{\dot{\alpha}\dot{\beta}}]=\frac{\partial ^2 \phi}{\partial w_{\alpha}w_{\beta}}. 
\end{equation*}

Then,

\begin{eqnarray} 
\hat{\Lambda}(x,A,B)&=&((\bar{\psi}\circ \psi^1) \circ \Lambda \circ (\psi^2)^{-1})(x,A,B) \nonumber \\
                    &=&((\bar{\psi} \circ \psi^1) \circ \Lambda)(j^2(\phi)) \nonumber \\
										&=& (\bar{\psi} \circ \psi^1) (j^1 \Phi) \nonumber \\
										&=& \bar{\psi}(j^1(\psi \circ \Phi)) \nonumber \\
										&=& ((\psi \circ \Phi)(0), \frac{\partial (\psi \circ \Phi)}{\partial u^{\alpha}}|_0 ).\label{lamda}									
\end{eqnarray}

From the last line of Equation \ref{lamda}, one can see that 

\begin{equation}
(\psi \circ \Phi)(0)=\psi(j^1 (\phi_0))=(\phi(0,0), \frac{\partial \phi(0,w)}{\partial w^{\alpha}}|_{w=0}) \in M\times L(\RR ^p, \RR ^m), \nonumber
\end{equation}

and 

\begin{equation}
\frac{\partial (\psi \circ \Phi)}{\partial u^{\alpha}}|_0=(\frac{\partial \phi (u,0)}{\partial u^{\alpha}}|_{u=0}, \frac{\partial ^2 \phi}{\partial u^{\alpha} \partial w^{\beta}}|_{(0,0)})  \in L(\RR^p, \RR ^m) \times L(\RR ^p, \RR^{mp}). \nonumber 
\end{equation}

Continuing to Equation (\ref{lamda}), we have

\begin{eqnarray*}
\hat{\Lambda}(x,A,B)&=&(\phi(0,0), \frac{\partial \phi(0,w)}{\partial w^{\alpha}}|_{w=0}, (\frac{\partial \phi (u,0)}{\partial u^{\alpha}}|_{u=0}, \frac{\partial ^2 \phi}{\partial u^{\alpha} \partial w^{\beta}}|_{u=w=0}) \nonumber \\
                    &=&(\phi(0,0),\frac{\partial \phi(u,w)}{\partial w^{\alpha}}|_{(u,w)=(0,0)}, \frac{\partial}{\partial u^{\alpha}}|_{u=0} (\phi(u,0), \frac{\partial \phi}{\partial w^{\beta}}|_{w=0}) \nonumber \\
										&=& (x, A_{\dot{\alpha}}, A_{\alpha}, B_{\alpha\dot{\beta}}).
\end{eqnarray*}

The map $(A,B)\to (x, A_{\dot{\alpha}}, A_{\alpha}, B_{\alpha\dot{\beta}})$ is a surjective linear map. It follows immediately that $\hat{\Lambda}$ is a submersion thus $\Lambda$ is.  

\end{proof}

\begin{remark} \label{remark1}

Using curve notation on each $j^2 \phi$ and $j^1 \Phi$ in $J^2_{2p} M$ and $\jp (\jp M)$ respectively are as follows:

\begin{equation*}
j^2 \phi=[x;[A_{\alpha} A_{\dot{\alpha}}]; B_{\bar{\alpha} \bar{\beta}}], 
\end{equation*}

considering the primary structure on $\jp (\jp M)$ we have

\begin{equation*}
j^1 \Phi =[(x,X_{\alpha}); (Y_{\alpha}, C_{\alpha \beta})],
\end{equation*}

where $(x,X_{\alpha})=\Phi(0)$, and $(Y_{\alpha}, C_{\alpha \beta})=\frac{\partial \Phi}{\partial u_{\alpha}}|_0$. 
\newline
  
Considering secondary structure on $\jp (\jp M)$, we have

\begin{equation*}
j^1 \Phi=[(x,X'_{\alpha}); (Y'_{\alpha}, C'_{\alpha \beta})] 
\end{equation*}
 
It follows from equations (\ref{secondary1}) and (\ref{secondary2}) that $X'_{\alpha}$, $Y'_{\alpha}$, and $C'_{\alpha \beta}$ are defined by the following:

\begin{equation} \nonumber
X'_{\alpha}=\frac{\partial (\pi \circ \Phi)}{\partial u_{\alpha}} , \hspace{2cm} Y'_{\alpha}=(pr_2 \circ \psi \circ \Phi)(0)=\frac{\partial \phi_0}{\partial u_{\alpha}}|_0, 
\end{equation} 

and 

\begin{equation} \nonumber
 C'_{\alpha \beta}= \frac{\partial (pr_2 \circ \psi \circ \Phi)}{\partial u_{\alpha}}|_0 = \frac{\partial ^2 \Phi}{\partial u_{\alpha} \partial w_{\beta}}|_0
\end{equation}

The local form of $\Lambda$ follows immediately that the curve notation (discussed in preliminaries section) is as following:

\begin{equation*}
\Lambda ([x;[A_{\alpha} A_{\dot{\alpha}}]; B_{\bar{\alpha}\bar{\beta}}])=[x, A_{\dot{\alpha}};A_{\alpha},B_{\alpha \dot{\beta}}] 
\end{equation*}

\end{remark}

Now we shall define an involution on second order jet bundles that will help us to define the involution of $\jp(\jp M)$.

\begin{theorem} \label{lemma4} 

There exist a canonical involution $\hat{\ell}: J^2_{2p} M \to J^2_{2p} M $ that descends via $\Lambda$ to $\jp(\jp M)$ giving the involution $\ell: \jp(\jp M) \to \jp(\jp M)$.

\end{theorem}

To prove this theorem, we need to show that $\hat{\ell}$ is a diffeomorphism, and the induced map $\ell$ is well defined. 

\begin{proof}

Let $\mathfrak{f}: \RR ^{2p} \to \RR^{2p}$ denote the flip map $(u,w) \to (w,u)$ (the map that flips first and second $p-$ tuples). Define 

\begin{equation}
\tilde{\ell}: J^2_{2p}M \to J^2_{2p}M \nonumber
\end{equation}

by

\begin{equation}
\tilde{\ell}(j^2\phi)=j^2(\phi \circ \mathfrak{f}) \nonumber
\end{equation}

The local form of $\tilde{\ell}$ is the involution $(x,A,B) \to (x,\bar{A}, \bar{B})$ such that

\begin{equation}
(x,\bar{A},\bar{B})=(x, [A_{\dot{\alpha}} A_{\alpha}], \begin{pmatrix} B_{\dot{\alpha}\dot{\beta}} & B_{\alpha \dot{\beta}} \\ B_{\dot{\alpha} \beta} & B_{\alpha \beta} \end{pmatrix}). \label{localform}
\end{equation}

(For detailed proof of Equation (\ref{localform}), see Appendix section, Corollary \ref{corollary8} ) 
\newline

The linear function corresponding to matrix $\bar{A}$ is

\begin{equation}
\bar{A}(U,W)=A_{\dot{\alpha}}U+ A_{\alpha} W=[A_{\alpha} A_{\dot{\alpha}}] .(W,U)=A(f(U,W)). \nonumber
\end{equation}
 
Therefore, $\bar{A}=A\circ f$. The symmetric bi-linear form corresponding to matrix $\bar{B}$ is

\begin{eqnarray}
\bar{B}((U,W),(\bar{U},\bar{W}))&=& (\bar{U}, \bar{W})^T.\bar{B}. (U,W) \nonumber \\
                                &=& B_{\dot{\alpha}\dot{\beta}} U \bar{U} + B_{\alpha\dot{\beta}} W \bar{U} + B_{\dot{\alpha}\beta} U \bar{W} + B_{{\alpha}\beta} W \bar{W} \nonumber \\
	 													    &=& \mathfrak{f} (\bar{U} , \bar{W} ) ^T . B . \mathfrak{f} ( U , W ) \nonumber \\
														    &=& B (\mathfrak{f} (U,W) , \mathfrak{f} (\bar{U} , \bar{W} ) ). \nonumber
\end{eqnarray}

Thus, $\bar{B}=B\circ (\mathfrak{f},\mathfrak{f})$. Then, the local form of $\tilde{\ell}$ is the function

\begin{equation}
(x,A,B) \to (x,A\circ \mathfrak{f}, B\circ (\mathfrak{f}, \mathfrak{f})) \nonumber
\end{equation}

which is a diffeomorphism. 
\newline

To finish the proof, we need to show that above defined $\ell$ is well defined.
\newline 

 Let $j^2\phi, j^2\phi' \in J^2_{2p}M.$
\newline

Suppose that $\Lambda(j^2\phi)=\Lambda(j^2\phi').$ Then,

\begin{eqnarray}
&&\widetilde{\psi}(\Lambda(j^2\phi))=\widetilde{\psi}(\Lambda(j^2\phi')) \nonumber \\
\Rightarrow && (\widetilde{\psi}\circ \Lambda)(j^2\phi)=(\widetilde{\psi}\circ \Lambda)(j^2\phi') \nonumber
\end{eqnarray}

By Diagram (\ref{lambdatilde}), we have

\begin{eqnarray}
(\hat{\Lambda} \circ \psi^2) (j^2\phi)&=&(\hat{\Lambda} \circ \psi^2)(j^2\phi') \nonumber \\
\Rightarrow (\phi(0,0), \frac{\partial\phi}{\partial w_{\alpha}}, \frac{\partial\phi}{\partial u_{\alpha}}, \frac{\partial^2\phi}{\partial u_{\alpha}\partial w_{\beta}})&=&(\phi'(0,0), \frac{\partial\phi'}{\partial w_{\alpha}}, \frac{\partial\phi'}{\partial u_{\alpha}}, \frac{\partial^2\phi'}{\partial u_{\alpha}\partial w_{\beta}}) \nonumber \\
\Rightarrow (\phi(0,0), A_{\dot{\alpha}}, A_{\alpha}, B_{\alpha \dot{\beta}})&=&(\phi'(0,0), A'_{\dot{\alpha}}, A'_{\alpha}, B'_{\alpha \dot{\beta}}) \nonumber
\end{eqnarray}

On the last equation, we can conclude that $B_{\alpha \dot{\beta}}=B'_{\alpha \dot{\beta}}$. Since  the matrices $B$ and $B'$ are symmetric, then 

\begin{equation*}
B_{\alpha \dot{\beta}}=[B_{\dot{\alpha}\beta}]^T
\end{equation*}

 which implies 

\begin{equation*}
B_{\dot{\alpha} \beta}=B'_{\dot{\alpha} \beta}.
\end{equation*}

Now, we focus on the function $\tilde{\ell}$. Using commutative Diagram (\ref{lambdatilde}), $\tilde{\psi} \circ \Lambda =\hat{\Lambda} \circ \psi ^2$. Then

\begin{eqnarray*}
(\tilde{\psi} \circ \Lambda \circ \tilde{\ell})(j^2\phi)&=&(\tilde{\psi} \circ \Lambda)(j^2(\phi \circ \mathfrak{f})) \nonumber \\
																												&=&(\hat{\Lambda} \circ \psi ^2) (j^2(\phi \circ \mathfrak{f})) \nonumber \\
																												&=& \hat{\Lambda} (\phi(0,0), \frac{\partial (\phi \circ \mathfrak{f})}{\partial \bar{u}_{\bar{\alpha}}}, \frac{\partial ^2 (\phi \circ \mathfrak{f})}{\partial \bar{u}_{\bar{\alpha}} \partial \bar{u}_{\bar{\beta}}}) \nonumber \\
				&=&\hat{\Lambda} (x,[A_{\dot{\alpha}},A_{\alpha}],\begin{pmatrix} B_{\dot{\alpha}\dot{\beta}} & B_{\alpha \dot{\beta}} \\ B_{\dot{\alpha} \beta} & B_{\alpha \beta} \end{pmatrix})\nonumber \\
																												&=& (x,A_{\alpha}, A_{\dot{\alpha}}, B_{\dot{\alpha} \beta})	\nonumber \\
																												&=& (x',A'_{\alpha}, A'_{\dot{\alpha}}, B'_{\dot{\alpha} \beta})	\nonumber \\
																												&=& (\tilde{\psi} \circ \Lambda \circ \tilde{\ell})(j^2 \phi '). \nonumber 
\end{eqnarray*}

On the other hand, because $\tilde{\psi}$ is injective, it follows immediately that 

\begin{equation}
(\Lambda \circ \tilde{\ell})(j^2 \phi)=(\Lambda \circ \tilde{\ell})(j^2 \phi '). \label{welldefined}
\end{equation}
 
Equation (\ref{welldefined}) implies that $\ell$ is a function on $\jp M$. Since $\Lambda$ is a surjective submersion, and $\tilde{\ell}$ is a diffeomorphism, then  $\ell$ is a differentiable function (see \cite{Clark} and \cite{Dieudonne}, 16.7.7,ii). Since $\tilde{\ell}$ is an involution i.e. $\tilde{\ell}$ has its own inverse, then $\ell$ is also an involution, and the function $\ell$ is a diffeomorphism.

\end{proof}

In the following proposition, we define aforementioned involution $\ell$.

\begin{proposition} \label{propell}

The involution $\ell:\jp(\jpm) \to \jp(\jpm)$ is defined by the following:

\begin{equation}
\ell([y,Y_{\alpha};X_{\alpha}, C_{\alpha \beta}])=[y, X_{\alpha}; Y_{\alpha}, [C_{\alpha \beta}]^T] \nonumber
\end{equation}

\end{proposition}

\begin{proof}

Suppose that $v=[y,Y_{\alpha}; X_{\alpha}, C_{\alpha \beta}] \in \jp(\jpm).$ Due to $\Lambda$ being a surjective function, there exists $j^2 \phi=[x,A_{\bar{\alpha}}, B_{\bar{\alpha} \bar{\beta}}] \in J^2_{2p} M$ such that 

\begin{equation}
A_{\bar{\alpha}}=[A_{\alpha} A_{\dot{\alpha}}], \nonumber 
\end{equation}

\begin{equation}
B_{\bar{\alpha} \bar{\beta}}=\begin{pmatrix} B_{\alpha \beta} & B_{\dot{\alpha}\beta} \\ B_{\alpha \dot{\beta}} & B_{\dot{\alpha} \dot{\beta}} \end{pmatrix}, \nonumber
\end{equation}

and  

\begin{equation*}
[y,Y_{\alpha}; X_{\alpha}, C_{\alpha \beta}] =\Lambda(j^2 \phi)=\Lambda ([x,A_{\bar{\alpha}}, B_{\bar{\alpha} \bar{\beta}}])=[x, A_{\dot{\alpha}}, A_{\alpha}, B_{\alpha \dot{\beta}}] 
\end{equation*}

Then,
\begin{equation}
x=y, \hspace{2.5cm} Y_{\alpha}=A_{\dot{\alpha}}, \hspace{2.5cm} X_{\alpha}=A_{\alpha}, \hspace{2.5cm} C_{\alpha \beta}= B_{\alpha \dot{\beta}}  \nonumber
\end{equation}

On the other hand, since the matrix $B_{\bar{\alpha}\bar{\beta}}$ is a symmetric, then $[B_{\alpha \dot{\beta}}]^T=B_{\dot{\alpha}\beta}$. Then

\begin{eqnarray*}
\ell([y,Y_{\alpha};X_{\alpha}, C_{\alpha \beta}])=(\ell \circ \Lambda)(j^2 \phi))&=& (\Lambda \circ \tilde{\ell})(j^2 \phi) \nonumber \\
                                       &=& [x, A_{\alpha}; A_{\dot{\alpha}}, B_{\dot{\alpha}\beta}]. \nonumber \\
																			 &=& [y, X_{\alpha}; Y_{\alpha}, [B_{\alpha \dot{\beta}}]^T] \nonumber \\
																			 &=& [y, X_{\alpha}; Y_{\alpha}, [C_{\alpha \beta}]^T]
\end{eqnarray*}

which finishes the proof.

\end{proof}

In the following proposition, we shall show that the primary and secondary vector bundle structures on $\jp(\jp M)$ are isomorphic.

\begin{proposition} 

For any smooth manifold $M$, the function

\begin{equation}
\ell : \jp (\jp M) \to \jp (\jp M) \nonumber
\end{equation}

which is defined in Theorem \ref{lemma4} is a $\jp M-$ isomorphism of vector bundles in the following way:

\begin{equation}
\begindc{\commdiag} 
\obj(1,4) [A] {$ \jp (\jpm) $}
\obj(5,4) [B] {$\jp (\jp M)$}
\obj(1,1) [C] {$ \jpm$}
\obj(5,1) [D] {$ \jpm $}
\mor{A}{B} {$\ell $}
\mor{A}{C} {$\tilde{\pi}$}
\mor{B}{D} {$\pi_M^1$}
\mor{C}{D} {$=$} 
\enddc \nonumber
\end{equation}
\end{proposition}

\begin{proof}

To prove the theorem, we only need to show that $\ell$ preserves fibers. 
\newline

Suppose that $j^1 \Phi \in \tilde{\pi}^{-1} ([x,X_{\alpha}])$. Then from Remark \ref{remark1}, $j^1 \Phi =[x,X_{\alpha}; Y_{\alpha}, C_{\alpha \beta}]$.
Using Proposition \ref{propell}, we have 

\begin{eqnarray}
 (\pi^1 \circ \ell)(j^1 \Phi) &=& (\pi^1 \circ \ell)([x,X_{\alpha}; Y_{\alpha}, C_{\alpha \beta}])   \nonumber \\
                              &=& \pi^1 ([x,Y_{\alpha}; X_{\alpha}, [C_{\alpha \beta}]^T) \nonumber \\
															&=& [x,X_{\alpha}] \nonumber \\
															&=& \tilde{\pi}(j^1 \Phi) \nonumber
\end{eqnarray}

which proves that $\ell$ preserves fibers. Thus $\ell$ is a bundle isomorphism.

\end{proof}

 
 \section{Appendix}

Below, we present some identities and statements that are used in this paper.

\begin{lemma}\label{theo1}

$(\jpm , \pi, M, L(\RR ^p, \RR ^m))$ is a vector bundle with the bundle projection $\pi_M:\jpm \to M$ which is defined by $\pi_M (j^1\phi)=\phi(0)$ for all $ j^1\phi \in \jpm $, the the above manifold structure on $\jpm $, and the intrinsic operations on $\pi_M^{-1}\{x\}$ defined by

\begin{equation}
[x;X_{\alpha}]+[x;X'_{\alpha}]=[x;X_{\alpha}+X'_{\alpha}] \hspace{2cm} \lambda \bullet [x;X_{\alpha}]=[x;\lambda X_{\alpha}] 
\end{equation}  

for all $x \in M$. 

\end{lemma}

\begin{proof}

In \cite{Cordero}, it is proven that $\jpm $ is a bundle with the local trivialization 

\begin{equation}
\Psi_M:J_p^1U \to U \times L(\RR ^p,\RR ^m) \nonumber
\end{equation} 

where $\jp U =\{j^1 \phi : \phi(0) \in U\}$. Showing that $(\Psi_M)_x: \pi_M^{-1}\{x\} \to L(\RR ^p,\RR ^m))$ being an isomorphism, implies that $(\jpm ,\pi, M, L(\RR ^p,\RR ^m)$ is a vector bundle.
\newline

From Equation \ref{equation1}, we have

\begin{eqnarray*}
(\Psi_M)_x(j^1\phi)=\frac{\partial(x_i \circ \phi)}{\partial u_{\alpha}}|_0
\end{eqnarray*}

If we denote $j^1\phi=[x;X_{\alpha}]$ where $X_{\alpha}=\phi^{\alpha}_*(d/du)$, then 

\begin{equation*}
(\Psi_M)_x(j^1\phi)=
\begin{pmatrix}
X_{\alpha}[x_i]
\end{pmatrix}.
\end{equation*}

Therefore, we have

\begin{eqnarray*}
(\Psi_M)_x([x;X_{\alpha}]+[x;X'_{\alpha}]) =(X_{\alpha}+X'_{\alpha})[x_i] &=& (X_{\alpha}[x_i]+ X'_{\alpha})[x_i]  \nonumber \\
                                                                          &=&(\Psi_M)_x([x;X_{\alpha}])+(\Psi_M)_x([x;X'_{\alpha}]).\nonumber \\
\end{eqnarray*}

On the other hand, we have

\begin{equation}
(\Psi_M)_x(\lambda \bullet [x;X_{\alpha}])= (\lambda X_{\alpha})[x_i] = \lambda  (X_{\alpha})[x_i] = \lambda (\Psi_M)_x([x;X_{\alpha}]) \nonumber
\end{equation}                                         
                                          
which shows that $(\Psi_M)_x$ is a linear function. 

\end{proof}   

\begin{lemma} 

The function $\Lambda$ defined in Lemma \ref{lemma4} is well defined.

\end{lemma}

\begin{proof}
 
Suppose that $j^2 \phi = j^2 \phi '$ , where  $j^2 \phi, j^2 \phi ' \in J_{2p}^2 (M) $. To prove the lemma, we should show that $j^1 \Phi=j^1\Phi'.$ 
\newline

 The equality of 2-jets implies followings:

\begin{equation} \phi (0,0)=\phi ' (0,0), \label{i1} \end{equation} 

\begin{equation}\frac{\partial \phi}{\partial \bar{u}_{\bar{\alpha}}} |_{(0,0)}= \frac{\partial \phi '}{\partial \bar{u}_{\bar{\alpha}}}|_{(0,0)}, \label{i02} \end{equation} 

\begin{equation} \frac{\partial ^2 \phi}{\partial \bar{u}_{\alpha}\bar{u}_{\beta}}|_{(0,0)} =\frac{\partial ^2 \phi '}{\partial \bar{u}_{\alpha}\bar{u}_{\beta}}|_{(0,0)}.  \label{i03} 
\end{equation}

Equation \ref{i1} shows that  

\begin{equation}
\phi_0(0)=\phi '_0(0). \label{iiii1}
\end{equation}

The Equation (\ref{i02}) leads 

\begin{equation}
\frac{\partial \phi}{\partial u_{\alpha}}|_{(0,0)} = \frac{\partial \phi '}{\partial u_{\alpha}}|_{(0,0)}\hspace{1cm} and  \hspace{1cm} \frac{\partial \phi}{\partial w_{\alpha}}|_{(0,0)} = \frac{\partial \phi '}{\partial w_{\alpha}}|_{(0,0)}.  \nonumber 
\end{equation} 

which implies 

\begin{equation}
\Rightarrow \frac{\partial \phi (u,0)}{\partial u_{\alpha}}|_{u=0} = \frac{\partial \phi '(u,0)}{\partial u_{\alpha}}|_{u=0} \label{ii1} 
\end{equation}

and

\begin{equation}
\Rightarrow \frac{\partial \phi(0,w)}{\partial w_{\alpha}}|_{w=0} = \frac{\partial \phi '(0,w)}{\partial w_{\alpha}}|_{w=0}.  \label{ii2}
\end{equation}

Equation (\ref{ii2}) leads 

\begin{equation}
\frac{\partial \phi_0}{\partial w_{\alpha}}= \frac{\partial \phi_0}{\partial w_{\alpha}}. \label{iiii2}
\end{equation}

Combining Equations (\ref{iiii1}) and (\ref{iiii2}), we have $$\Phi(0)=j^1(\phi_0)=j^1(\phi'_0)=\Phi'(0).$$

On the other hand, by setting $\bar{u}_{\bar{\alpha}}=u_{\alpha}$ and $\bar{u}_{\bar{\beta}}=w_{\beta}$ in Equation \ref{i03}, we have

\begin{equation}
\frac{\partial ^2 \phi}{\partial u_{\alpha} w_{\beta}}|_{(0,0)} =\frac{\partial ^2 \phi '}{\partial u_{\alpha} w_{\beta}}|_{(0,0)}. \label{iii1}
\end{equation}

For all $u \in \RR ^p$, $j^1(\phi_u)$ can be associated to the ordered pair $[\phi_u(0), \frac{\partial \phi_u}{\partial w_{\beta}}]$. Then, combining Equations \ref{ii1} and \ref{iii1}, we have

\begin{eqnarray}
\frac{\partial \Phi}{\partial u_{\alpha}}|_{u=0}&=&(\frac{\partial \phi_u(0)}{\partial u_{\alpha}}, \frac{\partial ^2 \phi}{\partial u_{\alpha} \partial w_{\beta}}|_{(0,0)}) \nonumber \\ 
																								&=& (\frac{\partial \phi(u,0)}{\partial u_{\alpha}}, \frac{\partial ^2 \phi}{\partial u_{\alpha} \partial w_{\beta}}|_{(0,0)}) \nonumber \\ 
																								&=&(\frac{\partial \phi '_u(0)}{\partial u_{\alpha}}, \frac{\partial ^2 \phi '}{\partial u_{\alpha} \partial w_{\beta}}|_{(0,0)}) \nonumber \\
																								&=& \frac{\partial \Phi'}{\partial u_{\alpha}}|_{u=0}
\end{eqnarray}

which shows that $j^1 \Phi =j^1 \Phi '$. Thus $\Lambda$ is well defined.

\end{proof}

\begin{lemma} 

Let $\phi \in C^{\infty}(\RR ^{2p}, M)$, and $u^{\alpha}$, $w^{\alpha}$ denote the first p coordinate functions and the second p coordinate functions respectively on $\RR ^{2p}$. Let $\mathfrak{f}$ be the flip map defined in Lemma \ref{lemma4}. Then there exists following equations:

\begin{eqnarray}
\frac{\partial (\phi \circ \mathfrak{f})}{\partial u_{\alpha}}|_0=\frac{\partial \phi}{\partial w_{\alpha}}|_0  \hspace{2cm}
\frac{\partial (\phi \circ \mathfrak{f})}{\partial w_{\alpha}}|_0=\frac{\partial \phi}{\partial u_{\alpha}}|_0  \label{first} \\
\frac{\partial ^2 (\phi \circ \mathfrak{f})}{\partial u_{\beta}\partial u_{\alpha}}|_0=\frac{\partial^2 \phi}{\partial w_{\beta} \partial w_{\alpha}}|_0  \label{eq3} \\
\frac{\partial ^2 (\phi \circ \mathfrak{f})}{\partial w_{\beta}\partial u_{\alpha}}|_0=\frac{\partial^2 \phi}{\partial u_{\beta} \partial w_{\alpha}}|_0  \label{eq4} \\
\frac{\partial^2 (\phi \circ \mathfrak{f})}{\partial w_{\beta} w_{\alpha}}|_0=\frac{\partial \phi}{\partial u_{\beta} u_{\alpha}}|_0  \label{eq5} 
\end{eqnarray}

\end{lemma}

\begin{proof}

Since $\mathfrak{f}(u,w)=(w,u)$, then 

\begin{equation}
u_{\alpha} \circ \mathfrak{f}=w_{\alpha}, \hspace{1cm} w_{\alpha} \circ \mathfrak{f}=u_{\alpha}. \label{coord}
\end{equation}

Due to the Chern Rule, we have

\begin{equation}
\frac{\partial (\phi \circ \mathfrak{f})}{\partial u_{\bar{\alpha}}}=(\frac{\partial \phi}{\partial u_{\bar{\beta}}} \circ \mathfrak{f}). \frac{\partial (u_{\bar{\beta}} \circ \mathfrak{f})}{\partial u_{\bar{\alpha}}}. \label{eq1}
\end{equation}

Combining \ref{coord} and \ref{eq1} at the point $0$, we obtain  Equation \ref{first}.
\newline

On the other hand, taking partial derivatives of two sides of Equation \ref{eq1} at the point $0$ and using Chern rule gives Equations \ref{eq3}, \ref{eq4} and \ref{eq5}.

\end{proof}

\begin{corollary} \label{corollary8}

Entries of the matrices $\bar{A}$ and $\bar{B}$ in Lemma \ref{lemma4} can be given as follows:

\begin{equation}
\bar{A}_{\alpha}=A_{\dot{\alpha}} \hspace{1cm} \bar{A}_{\dot{\alpha}}=A_{\alpha} \nonumber
\end{equation}

\begin{equation}
\bar{B}_{\alpha \beta}=B_{\dot{\alpha} \dot{\beta}} \hspace{1cm} \bar{B}_{\alpha \dot{\beta}}=B_{\dot{\alpha} \beta} \hspace{1cm} \bar{B}_{\dot{\alpha} \beta}=B_{{\alpha}\dot{\beta}} \hspace{1cm} \bar{B}_{\dot{\alpha}\dot{\beta}}=B_{\alpha\beta} \nonumber
\end{equation}

\end{corollary}

\begin{proof}

Using the definition of the matrix $\bar{B}$ and Equations (\ref{eq3}), (\ref{eq4}), and (\ref{eq5}) we have

\begin{equation}
[\bar{B}_{\alpha \beta}]= \frac{\partial ^2 (\phi \circ \mathfrak{f})}{\partial u_{\alpha} \partial u_{\beta}}|_0=\frac{\partial ^2 \phi }{\partial w_{\alpha} \partial w_{\alpha}}|_0=B_{\dot{\alpha} \dot{\beta}} \hspace{1cm}
[\bar{B}_{\dot{\alpha}\beta}]=\frac{\partial ^2 (\phi \circ \mathfrak{f})}{\partial w_{\alpha} \partial u_{\beta}}|_0=\frac{\partial ^2 \phi }{\partial u_{\alpha} \partial w_{\alpha}}|_0=B_{\alpha \dot{\beta}} \nonumber
\end{equation}

and

\begin{equation}
[\bar{B}_{\alpha \dot{\beta}}]= \frac{\partial ^2 (\phi \circ \mathfrak{f})}{\partial u_{\alpha} \partial w_{\beta}}|_0=\frac{\partial ^2 \phi }{\partial w_{\alpha} \partial u_{\beta}}|_0=B_{\alpha \dot{\beta}} \hspace{1cm}
[\bar{B}_{\dot{\alpha} \dot{\beta}}]= \frac{\partial ^2 (\phi \circ \mathfrak{f})}{\partial w_{\alpha} \partial w_{\beta}}|_0=\frac{\partial ^2 \phi }{\partial u_{\alpha} \partial u_{\beta}}|_0=B_{\alpha \beta}. \nonumber 
\end{equation}

\end{proof}

\begin{lemma} \label{lemmaomega}
The local form of the function $\Omega$ is 
$$\hat{\Omega}(u,f,y,g)=(u,y,f,g)$$.
\end{lemma}

\begin{proof}

For all $(u,f,y,g) \in \varphi_2 ( (\pie)^{-1}(\jp U))$, then there exists a unique $j^1 \Phi \in (\pie)^{-1}(\jp U)$ such that $\varphi_2(j^1 \Phi)=(u,f,y,g)$. Thus,

\begin{eqnarray*}
&& (u,f,y,g) \\
&=&((\varphi \times id_{ L(\RR^p, \RR^m) \times E \times L(\RR ^p, \RR ^k)}) \circ (\Psi_M \times \Psi_E) \circ \psi^1)(j^1 \Phi) \\
         &=& ((\varphi \times id_{ L(\RR^p, \RR^m) \times E \times L(\RR ^p, \RR ^k)}) \circ (\Psi_M \times \Psi_E)) (j^1 (\psi \circ \Phi)) \\
				 &=& ((\varphi \times id_{ L(\RR^p, \RR^m) \times E \times L(\RR ^p, \RR ^k)}) \circ (\Psi_M \times \Psi_E)) (j^1 (\pie \circ \Phi), j^1 (pr_2 \circ \psi \circ \Phi)) \\
				&=& (\varphi \times id_{ L(\RR^p, \RR^m) \times E \times L(\RR ^p, \RR ^k)})((\pie \circ \Phi)(0), (\frac{\partial (\pie \circ \Phi)}{\partial u_{\alpha}}|_0, (pr_2 \circ \psi \circ \Phi)(0), \frac{\partial (pr_2 \circ \psi \circ \Phi)}{\partial u_{\alpha}}|_0)) \\
				&=& ((\varphi \circ \pie \circ \Phi)(0), \frac{\partial (\pie \circ \Phi)}{\partial u_{\alpha}}|_0, (pr_2 \circ \psi \circ \Phi)(0), \frac{\partial (pr_2 \circ \psi \circ \Phi)}{\partial u_{\alpha}}|_0).
\end{eqnarray*}

Therefore, there exist the following equations:

\begin{equation*}
u=(\varphi \circ \pie \circ \Phi)(0), \hspace{1cm} f=\frac{\partial (\pie \circ \Phi)}{\partial u_{\alpha}}|_0, \hspace{1cm} y=(pr_2 \circ \psi \circ \Phi)(0), \hspace{1cm} g= \frac{\partial (pr_2 \circ \psi \circ \Phi)}{\partial u_{\alpha}}|_0. 
\end{equation*}

On the other hand,  $ \frac{\partial (\psi \circ \Phi)}{\partial u_{\alpha}}|_0: \RR^p \to \RR^{m+k}$. The derivative can be written of form

 \begin{equation*}
	\frac{\partial (\psi \circ \Phi)}{\partial u_{\alpha}}|_0=[\frac{\partial (pr_1 \circ \psi \circ \Phi)}{\partial u_{\alpha}}|_0 \hspace{1cm} \frac{\partial (pr_2 \circ \psi \circ \Phi)}{\partial u_{\alpha}}|_0 ].
\end{equation*}

Thus,  

\begin{equation*}
\xi (\frac{\partial (\psi \circ \Phi)}{\partial u_{\alpha}}|_0)= (\frac{\partial (pr_1 \circ \psi \circ \Phi)}{\partial u_{\alpha}}|_0 , \frac{\partial (pr_2 \circ \psi \circ \Phi)}{\partial u_{\alpha}}|_0 ).
\end{equation*}
 
Using this equation, we have

\begin{eqnarray*}
& &((\varphi \times id) \circ (\psi \times \xi) \circ \tilde{\psi})(j^1 \Phi) = ((\varphi \times id) \circ (\psi \times \xi))(\Phi(0), \frac{\partial (\psi \circ \Phi)}{\partial u_{\alpha}}) \\           
                                                                           &=& (\varphi \times id) ((pr_1 \circ \psi \circ \Phi)(0), (pr_2 \circ \psi \circ \Phi)(0),  \xi(\frac{\partial (\psi \circ \Phi)}{\partial u_{\alpha}})) \\
					                                                                 &=& ((\varphi \circ \pie \circ \Phi)(0), (pr_2 \circ \psi \circ \Phi)(0), (\frac{\partial (pr_1 \circ \psi \circ \Phi)}{\partial u_{\alpha}}|_0 , \frac{\partial (pr_2 \circ \psi \circ \Phi)}{\partial u_{\alpha}}|_0 ) \\
																																					&=& (u,y,f,g).																																				
\end{eqnarray*}

This proves that $\hat{\Omega}(u,f,y,g)=(u,y,f,g)$.

\end{proof}

\noindent {\bf{Acnowledgements.}} The Author would like to thank to Idaho State
University Department of Mathematics in particular Prof. Dr. Robert Fisher, Jr. for being such nice hosts during
her visit in Fall 2015. She also acknowledges library support from ISU.

\end{document}